\documentclass{amsart}
\usepackage{amsfonts}
\usepackage{amsmath,amssymb}
\usepackage{amsthm}
\usepackage{amscd}
\usepackage{graphics}
\usepackage{graphicx}
\theoremstyle{remark}{
\newtheorem{Def}{{\rm Definition}}
\newtheorem{Ex}{{\rm Example}}

}

\newtheorem{Prop}{Proposition}
\newtheorem{Thm}{Theorem}

\newtheorem{MThm}{Main Theorem}

\begin{document}
\title[New fold maps on 7-dimensional closed and simply-connected manifolds]{New explicit construction of fold maps on general 7-dimensional closed and simply-connected spin manifolds}
\author{Naoki Kitazawa}

\address{Institute of Mathematics for Industry, Kyushu University, 744 Motooka, Nishi-ku Fukuoka 819-0395, Japan\\
 TEL (Office): +81-92-802-4402 \\
 FAX (Office): +81-92-802-4405 \\
}
\email{n-kitazawa@imi.kyushu-u.ac.jp}
\urladdr{https://naokikitazawa.github.io/NaokiKitazawa.html}
\maketitle
\begin{abstract}
$7$-dimensional closed and simply-connected manifolds have been attractive as central and explicit objects in algebraic topology and differential topology of higher dimensional closed and simply-connected manifolds, which were studied actively especially in 1950s--60s.  

Attractive studies of the class of these $7$-dimensional manifolds were started by the discovery of so-called {\it exotic} spheres by Milnor. It has influenced on the understanding of higher dimensional closed and simply-connected manifolds via algebraic and abstract objects. Recently this class is still being actively studied via more concrete notions from algebraic topology such as concrete bordism theory by Crowley, Kreck, and so on. 

As a new kind of fundamental and important studies, the author has been challenging understanding the class in constructive ways via construction of {\it fold} maps, which are higher dimensional versions of Morse functions.
The present paper presents a new general method to construct ones on spin manifolds of the class. 

\end{abstract}
\section{Introduction.}
\label{sec:1}

Closed (and simply-connected) manifolds whose dimensions are $m \geq 5$ have been central objects in algebraic topology and differential topology around the 1950s--60s. They have been classified via algebraic and abstract objects. 
$7$-dimensional closed and simply-connected manifolds were explicit and central objects in this scene as the Milnor's discovery of 7-dimensional {\it exotic} spheres \cite{milnor} and a related work \cite{eellskuiper} show and the class has been attractive as \cite{crowleyescher}, \cite{crowleynordstrom}, \cite{kreck}, and so on, show: an {\it exotic} (homotopy) sphere is a homotopy sphere which is not diffeomorphic to any standard sphere.

\subsection{Terminologies and notation on singularities of differentiable maps and fold maps.}
\label{subsec:1.1}
A {\it singular} point $p \in X$ of a differentiable map $c:X \rightarrow Y$ between two differentiable manifolds is a point at which the rank of the differential ${dc}_p$ of the map is smaller than both the dimensions $\dim X$ and
 $\dim Y$. $S(c)$ denotes the set of all singular points of $c$ and this is the {\it singular set} of $c$. $c(S(c))$ is che {\it singular value set} of $c$. and $Y-c(S(c))$ is the {\it regular value set} of $c$. A {\it singular {\rm (}regular{\rm )} value} is a point in the singular (resp. regular) value set of $c$.

Throughout the present paper, manifolds and maps between manifolds are smooth (of class $C^{\infty}$) unless otherwise stated. 
\begin{Def}
\label{def:1}
Let $m>n \geq 1$ be integers. 
A smooth map $f$ from an $m$-dimensional smooth manifold with no boundary into an $n$-dimensional smooth manifold with no boundary is said to be a {\it fold} map if at each singular point $p$, the map has the form
$(x_1, \cdots, x_m) \mapsto (x_1,\cdots,x_{n-1},\sum_{k=n}^{m-i}{x_k}^2-\sum_{k=m-i+1}^{m}{x_k}^2)$ for suitable coordinates and an integer $0 \leq i(p) \leq \frac{m-n+1}{2}$.
\end{Def}

\begin{Prop}
In the situation of Definition \ref{def:1}, the following properties hold.
\label{prop:1}
\begin{enumerate}
\item For any singular point $p$, $i(p)$ is unique. 
\item The set consisting of all singular points of a fixed index of the map is a closed submanifold of dimension $n-1$ with no boundary of the $m$-dimensional manifold. 
\item The restriction map to the singular set is an immersion.
\end{enumerate}
\end{Prop}
We call {$i(p)$ in Proposition \ref{prop:1} the {\it index} of $p$. A {\it special generic} map is a fold map such that the index of each singular point is $0$.
 The class of special generic maps contains all Morse functions on closed manifolds with exactly two singular points, which are central objects in so-called Reeb's theorem, characterizing spheres topologically except the case where the manifold is $4$-dimensional. A $4$-dimensional closed manifold is diffeomorphic to $S^4$ if and only if it admits such a function. The class of special generic maps also contains all canonical projections of unit spheres. 
Fold maps have been fundamental and strong tools in studying algebraic topological, differential topological and more general algebraic or geometric properties of manifolds
 in the branch of the singularity theory of differentiable maps as Morse functions have been in so-called Morse theory.
\cite{thom} and \cite{whitney} are pioneering studies and on smooth maps on manifolds whose dimensions are equal to or greater than $2$ into the plane. After various studies, recently, Saeki, Sakuma and so on, have found interesting facts on fold maps satisfying appropriate conditions. Studies on special generic maps and manifolds admitting them are in \cite{saeki}, \cite{saeki2}, \cite{saekisakuma}, \cite{saekisakuma2}, \cite{sakuma}, \cite{sakuma2}, and so on. These studies also motivate the author to present \cite{kitazawa}, \cite{kitazawa2}, \cite{kitazawa3}, \cite{kitazawa4}, and so on,

\subsection{$7$-dimensional closed and simply-connected manifolds of several classes and explicit fold maps on them.} 
\label{subsec:1.2}
We assume that diffeomorphisms on manifolds are smooth. We define the {\it diffeomorphism group} of a manifold as the group of all diffeomorphisms on the manifold. Unless otherwise stated, the structure groups of bundles whose fibers are manifolds are subgroups of the diffeomorphism groups. In other words the bundles are {\it smooth} bundles.
\begin{Def}
\label{def:2}
For a fold map $f:M \rightarrow {\mathbb{R}}^n$ on a closed and connected manifold $M$, we also assume that $f {\mid}_{S(f)}$ is an embedding and that for each connected component $C$ of the singular value set and its small closed tubular neighborhood $N(C)$, the composition of $f {\mid}_{f^{-1}(N(C))}:f^{-1}(N(C)) \rightarrow N(C)$ with a canonical projection to $C$ gives a trivial bundle over $C$. In this situation we say that $f$ is {\it S-trivial}. 
\end{Def}
\begin{Thm}[\cite{kitazawa7} and \cite{kitazawa8}.]
\label{thm:1}
Let $A$, $B$ and $C$ be free commutative groups of rank $a$,$b$ and $c$.
Let $\{a_{i,j}\}_{j=1}^{a}$ be a sequence of integers where $1 \leq i \leq b$ is an integer. Let $p \in B \oplus C$.
Let $(h_{i,j})$ be a symmetric $b \times b$ matrix such that the $(i,j)$-th component is an integer and that the diagonal elements are $0$. In this situation, there exist a $7$-dimensional closed and simply-connected spin manifold and a fold map $f:M \rightarrow {\mathbb{R}^4}$ such that the following properties hold.
\begin{enumerate}
\item $H_{\ast}(M;\mathbb{Z})$ is free. Isomorphisms $H^2(M;\mathbb{Z}) \cong A \oplus B$ and $H^4(M;\mathbb{Z}) \cong B \oplus C$ hold and by fixing suitable identifications we have the following properties. 
\begin{enumerate}
\item Products of elements in $A \oplus \{0\} \subset H^2(M;\mathbb{Z})$ vanish.
\item Consider a suitable basis $\{({a_j}^{\ast},0)\}_{j=1}^{a}$ of $A \oplus \{0\} \subset H^2(M;\mathbb{Z})$ and a suitable basis $\{(0,{b_j}^{\ast})\}_{j=1}^{b}$ of $\{0\} \oplus B \subset H^2(M;\mathbb{Z})$. The product of $({a_{j_1}}^{\ast},0)$ and $(0,{b_{j_2}}^{\ast})$ is regarded as $(a_{j_2,j_1}{b_{j_2}}^{\ast},0) \in B \oplus \{0\} \subset H^4(M;\mathbb{Z})$.
The product of $(0,{b_{j_1}}^{\ast})$ and $(0,{b_{j_2}}^{\ast})$ is regarded as $(h_{j_1,j_2}{b_{j_1}}^{\ast}+h_{j_2,j_1}{b_{j_2}}^{\ast},0) \in H^4(M;\mathbb{Z})$.
\end{enumerate}
\item The 1st Pontryagin class of $M$ is $4p \in H^4(M;\mathbb{Z})$ where the identification before is considered.
\item The index of each singular point of $f$ is always at most $1$ and preimages of regular values are disjoint unions of at most $3$ copies of $S^3$. Furthermore, if $(h_{i,j})$ is the zero matrix, then we can construct this map $f$ as
 an S-trivial map such that preimages of regular values are disjoint unions of at most $2$ copies of $S^3$.
\end{enumerate} 

\end{Thm}

We explain several facts implying explicitly that for understanding classes of $7$-dimensional closed and simply-connected spin manifolds in more geometric and constructive ways, fold maps into ${\mathbb{R}}^4$ are interesting.

It is known that there exist exactly $28$ types of $7$-dimensional oriented homotopy spheres (see \cite{milnor} and see also \cite{eellskuiper}). 

Hereafter, for $p$ in the Euclidean space ${\mathbb{R}}^k$, $||p||$ denotes the distance between the origin $0$ and $p$ where the underlying space is endowed with the Euclidean metric. For positive integers $k$ and $r$, ${D_{\mathbb{N},r}}^k$ denotes the set $\{x \in {\mathbb{R}}^k \mid ||x|| \in \mathbb{N}, 1 \leq ||x|| \leq r\}$.

\begin{Thm}[\cite{kitazawa2} and so on.]
\label{thm:2}
Every $7$-dimensional homotopy sphere admits an S-trivial fold map $f$ into ${\mathbb{R}}^4$ satisfying the following properties. 

\begin{enumerate}
\item $f(S(f))={D_{\mathbb{N},3}}^4$.
\item The index of each singular point is always $0$ or $1$.
\item For each connected component of the regular value set of $f$, the preimage of a regular value is, empty, diffeomorphic to $S^3$, diffeomorphic to $S^3 \sqcup S^3$ and diffeomorphic to $S^3 \sqcup S^3 \sqcup S^3$, respectively.
\end{enumerate}
Moreover, we can show the following two facts.
\begin{enumerate}
\item $M$ admits an S-trivial fold map $f$ into ${\mathbb{R}}^4$ satisfying the second property of the previous three properties and the first property replaced by $f(S(f))={D_{\mathbb{N},1}}^4$ if and only if $M$ is a standard sphere.
\item $M$ admits an S-trivial fold map $f$ into ${\mathbb{R}}^4$ satisfying the second property of the previous three properties and the first property replaced by $f(S(f))={D_{\mathbb{N},2}}^4$ if and only if the homotopy sphere $M$ is oriented and one of 16 types of the 28 types, where the standard sphere is one of the 16 types. Furthermore, the third propertuy is replaced by the following: for each connected component of the regular value set of $f$, the preimage of a regular value is, empty, diffeomorphic to $S^3$, and diffeomorphic to $S^3 \sqcup S^3$, respectively.
\end{enumerate}
\end{Thm}
We review known results on special generic maps on homotopy spheres. As a specific case, if a homotopy sphere of dimension $7$ admits a special generic map into ${\mathbb{R}}^4$, then it is diffeomorphic to a standard sphere.
\begin{Thm}
\label{thm:3}
{\rm (}\cite{calabi}, \cite{saeki}, \cite{saeki2}, \cite{wrazidlo} and so on.{\rm )} Every exotic homotopy sphere of dimension $m>3$ admits no special generic map into ${\mathbb{R}}^{k}$ for $k=m-3,m-2,m-1$. Furthermore, $7$-dimensional oriented homotopy spheres of $14$ types of all the $28$ types admit no special generic map into ${\mathbb{R}}^3$.
\end{Thm}
\begin{Thm}[\cite{kitazawa7} and \cite{kitazawa8}.]
\label{thm:4}
In the situation of Theorem \ref{thm:1}, $M$ admits no special generic map into ${\mathbb{R}}^4$ when $p \in B \oplus C$ is not zero or there exists a non-zero element in $\{a_{i,j}\}_{j=1}^{a}$ or in $\{h_{i,j}\}$. 
\end{Thm}
\subsection{Main theorems--$7$-dimensional closed, simply-connected and spin manifolds of a new class and explicit fold maps on them--.} 
One of the the main theorem of the present paper is as follows.
\begin{MThm}
\label{mthm:1}
Let $A$ and $B$ be free commutative groups of rank $a$ and $b$.
Let $\{k_{j}\}_{j=1}^{a+b}$ be a sequence of integers such that integers in $\{k_{j+a}\}_{j=1}^{b}$ are $0$ ore $1$.
Let $Y_0$ be a $4$-dimensional closed and simply-connected spin manifold whose integral cohomology ring is isomorphic to that of a manifold represented as a connected sum of finitely many copies of $S^2 \times S^2$ and let $H^j$ denote $H^{j}(Y_0;\mathbb{Z})$.
In this situation, there exist a $7$-dimensional closed and simply-connected spin manifold $M$ and an S-trivial fold map $f:M \rightarrow {\mathbb{R}^4}$ satisfying the following properties.
\begin{enumerate}
\item
\label{mthm:1.1}
$H_{\ast}(M;\mathbb{Z})$ is free. 
\item 
\label{mthm:1.2}
There exist isomorphisms $H^2(M;\mathbb{Z}) \cong A \oplus H^2$, $H^3(M;\mathbb{Z}) \cong B \oplus H^2$, $H^4(M;\mathbb{Z}) \cong B \oplus H^2$ and $H^5(M;\mathbb{Z}) \cong A \oplus H^2$ and by fixing suitable identifications we have the following properties. 
\begin{enumerate}
\item
\label{mthm:1.2.1}
Products of elements in $A \oplus \{0\} \subset H^2(M;\mathbb{Z})$ vanish.
\item
\label{mthm:1.2.2}
Products of elements in $A \oplus \{0\} \subset H^2(M;\mathbb{Z})$ and $B \oplus \{0\} \subset H^3(M;\mathbb{Z})$ vanish.
\item
\label{mthm:1.2.3}
Consider a suitable basis $\{({a_j}^{\ast},0)\}_{j=1}^{a}$ for $A \oplus \{0\} \subset H^2(M;\mathbb{Z})$ and a suitable basis for $\{({b_j}^{\ast},0)\}_{j=1}^{b}$ for $B \oplus \{0\} \subset H^3(M;\mathbb{Z})$. We also take a suitable basis $\{(0,{h_j}^{\ast})\}_{j=1}^{{\rm rank} H^2}$ for $\{0\} \oplus H^2 \subset H^2(M;\mathbb{Z})$. The product of $({a_{j_1}}^{\ast},0)$ and $(0,{h_{j_2}}^{\ast})$ is regarded as $(k_{j_1}{h_{j_2}}^{\ast},0) \in \{0\} \oplus H^2 \subset H^4(M;\mathbb{Z})$. The product of $({b_{j_1}}^{\ast},0)$ and $(0,{h_{j_2}}^{\ast})$ is regarded as $(k_{a+j_1}{h_{j_2}}^{\ast},0) \in \{0\} \oplus H^2 \subset H^5(M;\mathbb{Z})$.
\item
\label{mthm:1.2.4}
For the suitable basis $\{(0,{h_j}^{\ast})\}_{j=1}^{{\rm rank} H^2}$ for $\{0\} \oplus H^2 \subset H^2(M;\mathbb{Z})$ just before, we have the following properties.
\begin{enumerate}
\item
\label{mthm:1.2.4.1} ${\rm rank} \quad H^2$ is even.

\item
\label{mthm:1.2.4.2}
We can take the basis so that the dual ${{\rm PD}_X({h_j}^{\ast})}^{\ast}$ of ${\rm PD}_X({h_j}^{\ast})$ is ${h_{\frac{{\rm rank} H^2}{2}+j}}^{\ast}$ for $1 \leq j \leq \frac{{\rm rank} H^2}{2}$. 
\item
\label{mthm:1.2.4.3}
For the suitable basis before, the product of $(0,{h_{j_1}}^{\ast}) \in \{0\} \oplus H^2 \subset H^2(M;\mathbb{Z})$ and $(0,{h_{j_2}}^{\ast}) \in \{0\} \oplus H^2 \subset H^2(M;\mathbb{Z})$ vanishes unless $|j_1-j_2|=\frac{{\rm rank} H^2}{2}$ and the product of $(0,{h_{j_1}}^{\ast}) \in \{0\} \oplus H^2 \subset H^2(M;\mathbb{Z})$ and $(0,{h_{j_2}}^{\ast}) \in \{0\} \oplus H^2 \subset H^2(M;\mathbb{Z})$ is ${\Sigma}_{j=1}^{b} (k_{a+j}{b_j}^{\ast},0) \in B \oplus \{0\} \cong H^4(M;\mathbb{Z})$ where $|j_1-j_2|=\frac{{\rm rank} H^2}{2}$.
\end{enumerate}
\end{enumerate}
\end{enumerate} 
\end{MThm}
For this class of manifolds, we can also show the following theorem.

\begin{MThm}
\label{mthm:2}
In Main Theorem \ref{mthm:1}, we can obtain manifolds which we cannot obtain in Theorem \ref{thm:1} under the constraint that the matrix $(h_{i,j})$ is the zero matrix.
\end{MThm}
\subsection{The content of the present paper.}
\label{subsec:1.3} 
The organization of the paper is as the following.
In the next section, we demonstrate construction of new fold maps on $7$-dimensional closed, simply-connected and spin manifolds and show Main Theorems (Theorems \ref{thm:5} and \ref{thm:6}). Key methods resemble methods in the referred articles in Theorems \ref{thm:1} and \ref{thm:4} in a sense. The methods are also mainly based on arguments in \cite{kitazawa5}, \cite{kitazawa6} and so on. It is an important fact that these two articles of the author are mainly on studies of topological properties of {\it Reeb spaces} of fold maps. The {\it Reeb space} $W_c$ of a map $c:X \rightarrow Y$ is the quotient space $X/{{\sim}_c}$ of $X$ where the equivalence relation ${\sim}_c$ on $X$ is as follows. For $x_1,x_2 \in X$, $x_1 {\sim}_c x_2$ holds if and only if they are in a same connected component of a same preimage $c^{-1}(y)$ ($y \in Y$).

Reeb spaces already appeared in \cite{reeb} for example. For a fold map, the Reeb space has been shown to be a polyhedron whose dimension is equal to that of the target in \cite{kobayashisaeki}, \cite{shiota}, and so on, and Reeb spaces have been fundamental tools in studying algebraic or geometric properties, especially, (algebraic) topological properties of the manifolds. Note also that investigating the homology groups, the cohomology rings and other algebraic invariants of the manifolds are different from investigating those of the Reeb spaces and more difficult. In considerable cases, Reeb spaces inherit information of these topological invariants of the manifolds admitting the map.
 
\section{Construction of new family of fold maps on $7$-dimensional closed, simply-connected and spin manifolds of a new class.}
\label{sec:2}
Hereafter, $M$ denotes a closed and connected manifold of dimension $m$, let $n<m$ be a positive integer and let  $f:M \rightarrow {\mathbb{R}}^n$ denote a smooth map unless otherwise stated. 
For a topological space $X$, which is regarded as a polyhedron in a canonical way, let $c$ be a homology class: a smooth manifold $X$ satisfies this. The class $c$ is {\it represented} by a closed, compact and PL submanifold $i(Y)$ with no boundary, if for a homology class ${\nu}_Y$ of degree $\dim Y$ canonically obtained from $Y$, $i_{\ast}({\nu}_Y)=c$ where $i:Y \rightarrow X$ denotes the inclusion: in other words ${\nu}_Y$ is the {\it fundamental class} if $Y$ is connected, orientable and oriented and also the generator of the homology group of degree $\dim Y$ respecting the orientation. 

The following special generic maps play important roles. Hereafter, a {\it linear} bundle means a smooth bundle whose fiber is a unit sphere or disc and whose structure group acts linearly in a standard way on the fiber.

\begin{Ex}
\label{ex:2}
Let $l \geq 0$ be an integer. Let $m>n\geq 2$ be integers. A closed and connected manifold $M$ of dimension $m$ represented as a connected sum of $l$ manifolds in a family $\{S^{l_j} \times S^{m-l_j}\}_{j=1}^{l}$ of products of exactly two  standard spheres satisfying $1 \leq l_j \leq n-1$ admits a special generic map $f:M \rightarrow {\mathbb{R}}^n$ satisfying the following properties.
\begin{enumerate}
\item $f {\mid}_{S(f)}$ is an embedding.
\item $f(M)$ is a compact submanifold and represented as a boundary connected sum of $l$ manifolds in the family $\{S^{l_j} \times D^{n-l_j}\}_{j=1}^{l}$ of product manifolds.
\item The following two submersions, or smooth maps with no singular point, give a trivial liner bundle whose fiber is $D^{m-n+1}$ and a smooth bundle whose fiber is $S^{m-n}$, respectively.
\begin{enumerate}
\item The composition of the restriction of $f$ to the preimage of a small collar neighborhood of $\partial f(M)$ with the canonical projection to $\partial f(M)$.
\item The restriction of $f$ to the preimage of the complementary set of the interior of the collar neighborhood
 before in $f(M)$.
\end{enumerate}
\item Let $f_M$ denote the surjection obtained by restricting the target of $f$ to $f(M)$. The homomorphism ${f_M}_{\ast}$ between the homology groups maps a class represented
 by $S^{l_j} \times \{{\ast}_{j,1}\} \subset S^{l_j} \times S^{m-l_j}$ in the connected sum to a class represented by $S^{l_j} \times \{{\ast}_{j,2}\} \subset S^{l_j} \times {\rm Int} D^{n-l_j} \subset S^{l_j} \times D^{n-l_j}$ in the boundary connected sum $f(M)$.
\end{enumerate}
\end{Ex}
The following proposition is a fundamental fact and rigorous proofs are left to readers.
\begin{Prop}
\label{prop:2}
In the situation of Example \ref{ex:2}, let $(m,n)=(7,4)$, $l>0$ and $l_a,l_b \geq 0$ be integers satisfying $l_a+l_b=l$ and let $l_j=2$ for $1 \leq j \leq l_a$ and $l_j=3$ for $l_a+1 \leq j \leq l$. 
The homology class ${\nu}_j$ of the manifold is the class represented
 by $S^{l_j} \times \{{\ast}_{j,1}\} \subset S^{l_j} \times S^{m-l_j}$ in the connected sum. We have the following two properties.
\begin{enumerate}
\item For a copy $X$ of $S^3$, put a generator ${\nu}_{X,3}$ of its
3rd integral homology group, isomorphic to $\mathbb{Z}$, or a fundamental class.
Let $\{a_{j+l_a}\}_{j=1}^{l_b}$ be a sequence of integers of length $l_b$ such that the integers are $0$ or $1$.
In this situation, there exists an embedding $i_{X,f(M)}$ of $X$ into the interior of $f(M)$ such that ${i_{X,f(M)}}_{\ast}({\nu}_{X,3})={\Sigma}_{j=1}^{l_b} a_{j+l_a} {f_M}_{\ast}({\nu}_{j+l_a})$.
\item For a copy $X$ of $S^2 \times S^1$, take a generator ${\nu}_{X,2}$ of its 2nd integral homology group, isomorphic to $\mathbb{Z}$, and a generator ${\nu}_{X,3}$ of its
3rd integral homology group, isomorphic to $\mathbb{Z}$, or a fundamental class. Let $\{a_j\}_{j=1}^l$ be a sequence of integers of length $l$ such that the integers in $\{a_{j+l_a}\}_{j=1}^{l_b}$ are $0$ or $1$.
In this situation, there exists an embedding $i_{X,f(M)}$ of $X$ into the interior of $f(M)$ such that ${i_{X,f(M)}}_{\ast}({\nu}_{X,2})={\Sigma}_{j=1}^{l_a} a_j{f_M}_{\ast}({\nu}_j)$ and that ${i_{X,f(M)}}_{\ast}({\nu}_{X,3})={\Sigma}_{j=1}^{l_b} a_{j+l_a} {f_M}_{\ast}({\nu}_{j+l_a})$.
\end{enumerate}
\end{Prop}

For a compact manifold $X$, let there exist a closed and connected manifold $X_0$ such that $X$ is obtained by removing the interior of the union of two smoothly and disjointly embedded unit discs of $\dim X_0$ and a Morse function $c:X_0 \rightarrow \mathbb{R}$ such that at distinct singular points the values are distinct, that there exist exactly two local extrema $a<b$, that their preimages are in the two embedded unit discs and that on the distinctly embedded unit discs there exists no other singular point. Let $c_{X,(X_0,X)}$ denote the restriction $c {\mid}_X$.

For a graded commutative algebra $A$ over $\mathbb{Z}$, we define the {\it $i$-th module} as the module consisting of all elements of degree $i$ of $A$. We also assume that the $0$-th module is isomorphic to $\mathbb{Z}$.

For a non-negative integer $i \geq 0$, we define the {\it $\leq i$-part} $A_{\leq i}$ of $A$ as a graded commutative algebra over $\mathbb{Z}$ as the following and as a graded module, this is regarded as a submodule of the module $A$.
\begin{enumerate}
\item The $j$-th module is same as that of $A$ for any $j \leq i$.
\item The product of two elements such that the sum of the degrees is smaller than or equal to $i$ is same as that in the case of $A$.
\item The $j$-th module is the zero module for any $j>i$.
\end{enumerate}
\begin{Prop}
\label{prop:3}
In the situation just before, let $X_0$ be orientable satisfying $\dim X_0>1$ and let $i_{X}:X \rightarrow X_0$ denote the inclusion.
\begin{enumerate} 
\item The restriction of ${i_{X}}^{\ast}:H^{\ast}(X_0;\mathbb{Z}) \rightarrow H^{\ast}(X;\mathbb{Z})$ to $H^{\ast}(X_0;\mathbb{Z})_{\leq \dim X-2}$ is an isomorphism between the graded commutative algebras $H^{\ast}(X_0;\mathbb{Z})_{\leq \dim X-2}$ and $H^{\ast}(X;\mathbb{Z})_{\leq \dim X-2}$.
\item The restriction of ${i_{X}}^{\ast}:H^{\ast}(X_0;\mathbb{Z}) \rightarrow H^{\ast}(X;\mathbb{Z})$ to $H^{\ast}(X_0;\mathbb{Z})_{\leq \dim X-1}$ gives a monomorphism between the graded commutative algebras $H^{\ast}(X_0;\mathbb{Z})_{\leq \dim X-1}$ and $H^{\ast}(X;\mathbb{Z})_{\leq \dim X-1}$ and $H^{\ast}(X;\mathbb{Z})_{\leq \dim X-1}$ is represented as the internal direct sum of the image of the monomorphism and a commutative subgroup G of $H^{\ast}(X;\mathbb{Z})_{\leq \dim X-1}$ isomorphic to $\mathbb{Z}$.
\end{enumerate}
\end{Prop}
We need several notions and explain them. For a closed, connected  and oriented manifold $X$, we denote the so-called {\it Poincar\'e dual} to an integral (co)homology class $c$ by ${\rm PD}_X(c)$.
If for a (compact) topological space $X$ whose integral homology group is free, then we can define the {\it dual} $c^{\ast} \in H^{j}(X;\mathbb{Z})$ of a homology class $c \in H_{j}(X;\mathbb{Z})$ we cannot represent as
 a form $kc^{\prime}$ such that $k \neq 0,1,-1$ is an integer and that $c^{\prime}$ is not zero uniquely. This is the element satisfying the following properties.
\begin{enumerate}
\item $c^{\ast}(c)=1$.
\item For any subgroup $G$ of $H_{j}(X;\mathbb{Z})$ making $H_{j}(X;\mathbb{Z})$ the internal direct sum of the group generated by $c$ and $G$ and for any $g \in G$, $c^{\ast}(g)=0$.
\end{enumerate}
Hereafter, $\cong$ between groups means that the groups are isomorphic.
\begin{Thm}
\label{thm:5}
Let $A$ and $B$ be free commutative groups of rank $a$ and $b$.
Let $\{k_{j}\}_{j=1}^{a+b}$ be a sequence of integers such that integers in $\{k_{j+a}\}_{j=1}^{b}$ are $0$ ore $1$.
Let $Y_0$ be a $4$-dimensional closed, simply-connected and spin manifold whose integral cohomology ring is isomorphic to that of a manifold represented as a connected sum of finitely many copies of $S^2 \times S^2$ and let $H^j$ denote $H^{j}(Y_0;\mathbb{Z})$. In this situation, there exist a $7$-dimensional closed and simply-connected spin manifold $M$ and an S-trivial fold map $f:M \rightarrow {\mathbb{R}^4}$ satisfying the following properties.
\begin{enumerate}
\item
\label{thm:5.1}
$H_{\ast}(M;\mathbb{Z})$ is free.
\item
\label{thm:5.2}
There exist isomorphisms $H^2(M;\mathbb{Z}) \cong A \oplus H^2$, $H^3(M;\mathbb{Z}) \cong B \oplus H^2$, $H^4(M;\mathbb{Z}) \cong B \oplus H^2$ and $H^5(M;\mathbb{Z}) \cong A \oplus H^2$ and by fixing suitable identifications we have the following properties.
\begin{enumerate}
\item
\label{thm:5.2.1}
Products of elements in $A \oplus \{0\} \subset H^2(M;\mathbb{Z})$ vanish.
\item
\label{thm:5.2.2}
Products of elements in $A \oplus \{0\} \subset H^2(M;\mathbb{Z})$ and $B \oplus \{0\} \subset H^3(M;\mathbb{Z})$ vanish.
\item
\label{thm:5.2.3}
Consider a suitable basis $\{({a_j}^{\ast},0)\}_{j=1}^{a}$ for $A \oplus \{0\} \subset H^2(M;\mathbb{Z})$ and a suitable basis $\{({b_j}^{\ast},0)\}_{j=1}^{b}$ for $B \oplus \{0\} \subset H^3(M;\mathbb{Z})$. We also take a suitable basis $\{(0,{h_j}^{\ast})\}_{j=1}^{{\rm rank} H^2}$ for $\{0\} \oplus H^2 \subset H^2(M;\mathbb{Z})$. The product of $({a_{j_1}}^{\ast},0)$ and $(0,{h_{j_2}}^{\ast})$ is regarded as $(0,k_{j_1}{h_{j_2}}^{\ast}) \in \{0\} \oplus H^2 \subset H^4(M;\mathbb{Z})$. The product of $({b_{j_1}}^{\ast},0)$ and $(0,{h_{j_2}}^{\ast})$ is regarded as $(0,k_{a+j_1}{h_{j_2}}^{\ast}) \in \{0\} \oplus H^2 \subset H^5(M;\mathbb{Z})$.
\item
\label{thm:5.2.4}
For the suitable basis $\{(0,{h_j}^{\ast})\}_{j=1}^{{\rm rank} H^2}$ for $\{0\} \oplus H^2 \subset H^2(M;\mathbb{Z})$ just before, we have the following properties.
\begin{enumerate}
\item
\label{thm:5.2.4.1} ${\rm rank} \quad H^2$ is even.
\item
\label{thm:5.2.4.2}
We can take the basis so that the dual ${{\rm PD}_{Y_0}({h_j}^{\ast})}^{\ast}$ of ${\rm PD}_{Y_0}({h_j}^{\ast})$ is ${h_{\frac{{\rm rank} H^2}{2}+j}}^{\ast}$ for $1 \leq j \leq \frac{{\rm rank} H^2}{2}$. 
\item
\label{thm:5.2.4.3}
For the suitable basis before, the product of $(0,{h_{j_1}}^{\ast}) \in \{0\} \oplus H^2 \subset H^2(M;\mathbb{Z})$ and $(0,{h_{j_2}}^{\ast}) \in \{0\} \oplus H^2 \subset H^2(M;\mathbb{Z})$ vanishes unless $|j_1-j_2|=\frac{{\rm rank} H^2}{2}$
and the product of $(0,{h_{j_1}}^{\ast}) \in \{0\} \oplus H^2 \subset H^2(M;\mathbb{Z})$ and $(0,{h_{j_2}}^{\ast}) \in \{0\} \oplus H^2 \subset H^2(M;\mathbb{Z})$ is ${\Sigma}_{j=1}^{b} (k_{a+j}{b_j}^{\ast},0) \in B \oplus \{0\} \cong H^4(M;\mathbb{Z})$ where $|j_1-j_2|=\frac{{\rm rank} H^2}{2}$.
\end{enumerate}
\end{enumerate}
\end{enumerate} 
\end{Thm}
\begin{proof}
First we construct a special generic map $f_0$ on an $7$-dimensional manifold $M_0$ into ${\mathbb{R}}^4$ in Example \ref{ex:2} by setting $(l_a,l_b)=(a,b)$ in Proposition \ref{prop:2}. We replace $a_j$ by $k_j$ there. Take a copy ''$X$ in Proposition \ref{prop:2}'' and remove the interior of a small closed tubular
neighborhood $N(X)$ and its preimage.

Replace this removed map by a product map of $c_{Y,(Y_0,Y)}$ before where ''$X$'' and ''$X_0$'' are replaced by a suitable compact manifold $Y$ and $Y_0$ with a suitable Morse function $c$ respectively and the identity map on a manifold diffeomorphic to ''$X$ in Proposition \ref{prop:2} here''. Note that this is regarded as a finite iteration of normal bubbling operations in \cite{kitazawa5}, \cite{kitazawa6}, and so on. Thus we have a desired fold map $f$.

We observe the integral homology group and the integral cohomology ring of the manifold $M$ to complete the proof.

First we show the first statement (\ref{thm:5.1}) and the second statement (\ref{thm:5.2}) on the cohomology group. $H_2(M_0;\mathbb{Z})$ and $H_3(M_0;\mathbb{Z})$ are generated by classes represented by standard spheres in $M_0$ and apart from ${f_0}^{-1}(N(X))$ and from these standard spheres and homology classes represented by them we have subgroups isomorphic to $H_2(M_0;\mathbb{Z})$ and  $H_3(M_0;\mathbb{Z})$ respectively with their bases.
Let $A^{\ast}:=H^2(M_0;\mathbb{Z}) \cong A$ and $B^{\ast}:=H^3(M_0;\mathbb{Z}) \cong B$: we take the duals of classes in the bases. We take a basis of $H_2(Y;\mathbb{Z}) \cong H_2(Y_0;\mathbb{Z})$ such that we can define the dual of each of the elements in the basis. 

We go to (\ref{thm:5.2.4}) and will show the first two statements. Since $Y_0$ is a $4$-dimensional, closed, simply-connected and spin manifold whose integral cohomology ring is isomorphic to that of a manifold represented as a connected sum of finitely many copies of $S^2 \times S^2$, the first two statements hold (for $H^2$). Furthermore, we can obtain a basis. For the basis of $H_2(Y;\mathbb{Z}) \cong H_2(Y_0;\mathbb{Z})$ for each of which we can define the dual as before, we can take their duals.

We will show on the existence of identifications of the cohomology groups with groups represented as direct sums of given abstract groups or $H^2$. 
By considering elements of the bases of $A^{\ast} \cong A$ and $H^2$ as integral cohomology classes of $M$ in a canonical way, we can observe the structure of $H^2(M;\mathbb{Z})$ and obtain a desired identification. $H^4(M;\mathbb{Z})$ is isomorphic to the internal direct sum of the following two subgroups $A_{4,1}$ and $A_{4,2}$.
\begin{enumerate}
\item The classes represented by standard spheres in $M_0$ and apart from ${f_0}^{-1}(N(X))$ and  forming a basis of $H_3(M_0;\mathbb{Z})$. Consider the classes represented by these spheres in $M$ and their Poincar\'e duals. We define $A_{4,1}$ as the subgroup generated by all the Poincar\'e duals. This is isomorphic to $B$.
\item We take the image of a section of the trivial bundle defined by the restriction of $f_0$ to ${f_0}^{-1}(X)$. Each element of the basis of $H_2(Y_0;\mathbb{Z}) \cong H_2(Y;\mathbb{Z})$ used to obtain the basis for $H^2$ before is represented by a 2-dimensional closed submanifold with no boundary and we can obtain a product of the image of the section over $S^2 \times \{\ast\} \subset X=S^2 \times S^1$ and this submanifold in $M$ in a canonical way. We can take the dual of the $4$-th integral homology class of $M$ represented by the resulting $4$-dimensional submanifold. $A_{4,2}$ is defined as a subgroup generated by the
set of all such classes. This is isomorphic to $H^2$.
\end{enumerate}
Thus we also have a desired identification. We can observe the structure of $H^3(M;\mathbb{Z})$ and obtain a desired identification by the following two steps.
\begin{enumerate}
\item Observe elements of the basis of $B^{\ast} \cong B$.
\item Regard $H^2$, isomorphic to $A_{4,2}$, as the image of the map corresponding each element of the basis of the submodule used to obtain $A_{4,2}$ just before to the Poincar\'e dual.
\end{enumerate} 
$H^5(M;\mathbb{Z})$ is isomorphic to the internal direct sum of the following two subgroups $A_{5,1}$ and $A_{5,2}$.
\begin{enumerate}
\item The classes represented by standard spheres in $M_0$ and apart from ${f_0}^{-1}(N(X))$ and  forming a basis for $H_2(M_0;\mathbb{Z})$. Consider the classes represented by these spheres in $M$ and their Poincar\'e duals. We define $A_{5,1}$ as the subgroup generated by all the Poincar\'e duals. This is isomorphic to $A$.
\item We take the image of a section of the trivial bundle defined by the restriction of $f_0$ to ${f_0}^{-1}(X)$. Each element of the suitable generator of $H_2(Y_0;\mathbb{Z}) \cong H_2(Y;\mathbb{Z})$ used to obtain the basis for $H^2$ before is represented by a 2-dimensional closed submanifold with no boundary and we can obtain a product of the image of the section over $X=S^2 \times S^1$ and this submanifold in $M$ in a canonical way. We can take the dual as a 5-th integral cohomology class of $M$. $A_{5,2}$ is defined as a subgroup generated by the set of all such classes. This is isomorphic to $H^2$.
\end{enumerate}
Thus we also have a desired identification. We have desired identifications. 

Note also that we can demonstrate this construction so that $M$ is simply-connected and spin.
We can show (\ref{thm:5.2.1}) and (\ref{thm:5.2.2}) immediately by virtue of the structures of the cohomology groups, submanifolds which are used in the important bases as homology classes represented by them and (\ref{thm:5.2.3}) follows by virtue of the structures of the cohomology groups, submanifolds used in the important bases as homology classes represented by them and Proposition \ref{prop:2}. The third statement of (\ref{thm:5.2.4}) follows by virtue of the structures of the cohomology groups, submanifolds used in the important bases as homology classes represented by them and Propositions \ref{prop:2} and \ref{prop:3}. 

This completes the proof.

\end{proof}

\begin{Thm}
\label{thm:6}
In Theorem \ref{thm:5}, we have manifolds $M$ having integral cohomology rings we cannot obtain as the integral cohomology rings of the manifolds in Theorem \ref{thm:1} where the matrix $(h_{i,j})$ is the  zero matrix.
\end{Thm}
\begin{proof}
In Theorem \ref{thm:5}, set $a=1$, $b=1$, $k_1 \neq 0$, $k_2 \neq 0$ and  $Y_0:=S^2 \times S^2$. We can easily see that $H^j(M;\mathbb{Z})$ is free and of rank $3$ for $j=2,3,4,5$. We cannot take a
submodule of rank $2$ of $H^2(M;\mathbb{Z})$ consisting of elements such that the squares are zero.

In Theorem \ref{thm:1} under the constraint that the matrix $(h_{i,j})$ is the zero matrix, we consider a $7$-dimensional closed and simply-connected manifold such that $H^2(M;\mathbb{Z})$ is free and of rank $3$. Products of elements in $A \oplus \{0\} \subset H^2(M;\mathbb{Z})$ and products of elements in $\{0\} \oplus B \subset H^2(M;\mathbb{Z})$ vanish. This implies that we can take a
submodule of rank $2$ of $H^2(M;\mathbb{Z})$ consisting of elements such that the squares are zero.

This completes the proof.
\end{proof}
\section{Acknowledgement.}
\label{sec:3}
\thanks{The author is a member of the project JSPS KAKENHI Grant Number JP17H06128 "Innovative research of geometric topology and singularities of differentiable mappings"
(https://kaken.nii.ac.jp/en/grant/KAKENHI-PROJECT-17H06128/).

 Independently, this work was supported by "The Sasakawa Scientific Research Grant" (2020-2002 : https://www.jss.or.jp/ikusei/sasakawa/).}

\end{document}